\documentclass[11pt]{amsart}    
\usepackage{graphicx}

\usepackage{amsmath,amssymb,latexsym} 
\usepackage{graphicx}
\usepackage{centernot}
\usepackage{fullpage}
\usepackage[square, numbers]{natbib}   
\usepackage{comment}
\usepackage{mathrsfs}

\usepackage{enumitem}

\newtheorem{theorem}{Theorem} 

\newtheorem{alphtheorem}{Theorem}

\newtheorem{lemma}{Lemma}
\theoremstyle{definition}

\newtheorem{conjecture}{Conjecture}
\theoremstyle{remark}

\def\Z{\mathbb{Z}}

\def\F{\mathcal{F}}

\def\zero{\boldsymbol{0}}

\def\ds{\displaystyle}
\def\KG{\operatorname{KG}}

\def\cd{\operatorname{cd}}
\def\ecd{\operatorname{ecd}}

\def\alt{\operatorname{alt}}

\def\HH{\mathcal{H}}

\def\sd{\operatorname{sd}}

\title{Coloring Properties of Categorical Product of General  Kneser Hypergraphs}
\author{Roya Abyazi Sani}
\author{Meysam Alishahi}
\author{Ali Taherkhani}

\address{R. Abyazi Sani, 
School of Mathematical Sciences,
Shahrood University of Technology, Shahrood, Iran}
\email{r.abyazi@gmail.com}
\address{M. Alishahi, 
School of Mathematical Sciences,
Shahrood University of Technology, Shahrood, Iran}
\email{meysam\_alishahi@shahroodut.ac.ir}
\address{A. Taherkhani, 
Department of Mathematics, Institute for Advanced Studies in Basic Sciences (IASBS), Zanjan 45137-66731, Iran}
\email{ali.taherkhani@iasbs.ac.ir}

\begin{document}
\maketitle

\begin{abstract} 
\noindent 
More than 50 years ago Hedetniemi conjectured that the chromatic number of categorical product of two graphs is equal to the minimum of their chromatic numbers.
This conjecture has received a considerable attention in recent years.  Hedetniemi's conjecture were generalized to hypergraphs by Zhu in 1992.
Hajiabolhassan and Meunier (2016) introduced the first nontrivial lower bound for the chromatic number of categorical product of general Kneser hypergraphs and  using 
this lower bound, they verified Zhu's conjecture for some families of hypergraphs. In this paper, we shall present some colorful type results for the coloring of categorical product of general Kneser hypergraphs, which generalize  the Hajiabolhassan-Meunier result. 
Also, we present a new lower bound for the chromatic number of categorical product of general Kneser hypergraphs which can be extremely better than the Hajiabolhassan-Meunier lower bound.
Using this lower bound, we enrich the family of hypergraphs satisfying Zhu's conjecture. \\

\end{abstract}

\section{\bf Introduction and Main Results}
For two graphs $G$ and $H$, their categorical product $G\times H$ is the  graph defined 
on the vertex set $V(G)\times V(H)$ such that two vertices $(g,h)$ and $(g',h')$ are adjacent whenever 
$gg'\in E(G)$ and $hh'\in E(H)$. The categorical product is the product involved in 
the famous long-standing conjecture posed by Hedetniemi. 
Hedetniemi's conjecture states that the chromatic number of $G\times H$ is equal to 
the minimum of $\chi(G)$ and $\chi(H)$. It was shown that the conjecture is true for several families 
of graphs but it is wide open (see, Tardif~\cite{MR2445666} and Zhu~\cite{MR1609464}). 
In spite of being investigated in several articles, there is no fascinating progress in solving 
this conjecture. This conjecture was generalized to the case of hypergraphs in~\cite{MR1206546}. 

A hypergraph $\HH$ is an ordered pair $(V(\HH),E(\HH))$  where  $V(\HH)$ is a set of vertices, and $E(\HH)$ is a family of nonempty subsets of $V(\HH)$.
The elements of $E(\HH)$ are called hyperedges. A hypergraph $\HH$ is said to be $r$-uniform if $E(\HH)$  is a family of distinct $r$-subsets of $V(\HH)$. In particular, a $2$-uniform hypergraph is   called a graph. An $r$-uniform hypergraph $\F$ is called {\it a complete $r$-partite hypergraph} if $V(\F)$ 
can be partitioned into $r$ parts (subsets) $V_{_{1}},\ldots,V_r$ such that 
the edge set of $\F$ is the set of all $r$-subsets of $V$ intersecting each part $V_{_{i}}$ in 
exactly one vertex.  The hypergraph $\F$ is said to be 
{\it balanced} if  $|V_{_{i}}|-|V_{_{j}}|\leq 1$ for each $i,j\in[r]$. Also, for an $r$-uniform hypergraph $\F$ and pairwise disjoint subsets $U_{_{1}},\ldots,U_r\subseteq V(\F)$, 
the hypergraph $\F[U_{_{1}},\ldots,U_r]$ is defined to be a subhypergraph of 
$\F$ whose vertex set  is $\bigcup\limits_{i=1}^r U_{_{i}}$ and  whose edge set consists of all hyperedges of $\F$ which have exactly one  element in each $U_{_{i}}$.

{\it A proper coloring} of a hypergraph $\HH$ is  an assignment of colors to vertices of $\HH$ such that there is no monochromatic hyperedge. 
The chromatic number of a hypergraph $\HH$,  denoted by $\chi(\HH)$, is the smallest number $k$ such that there exists  a proper coloring  of $\HH$ with $k$ colors. 
If there is no such a $k$, we define the chromatic number to be infinite. 
Let $c$ be a proper coloring of a complete $r$-partite hypergraph $\F$
with parts $V_{_{1}},\ldots,V_r$. The hypergraph $\F$ is {\it colorful} 
(with respect to the coloring $c$)
whenever for each $i\in [r]$, the vertices in $V_{_{i}}$ receive different colors, that is, 
$|c(V_{_{i}})|=|V_{_{i}}|$ for each $i\in [r]$. 

Let $\HH_{_{1}}=(V_{_{1}},E_{_{1}})$ and $\HH_{_{2}}=(V_{_{2}},E_{_{2}})$ be two hypergraphs. 
For $i=1,2$, the projection $\pi_{_{i}}$ is defined by $\pi_{_{i}}(v_{_{1}},v_{_{2}})\mapsto v_{_{i}}$.
The categorical product of two hypergraphs $\HH_{_1}$ and $\HH_{_2}$ is the hypergraph $\HH_{_{1}}\times \HH_{_{2}}$ with vertex set $V_{_{1}}\times V_{_{2}}$ and hyperedge set
$$\{e\subseteq V_{_{1}}\times V_{_{2}}\colon \pi_{_{1}}(e)\in E_{_{1}}, \pi_{_{2}}(e)\in E_{_{2}}\}.$$
The categorical product of two hypergraphs is defined by D{\"o}rfler and Waller~\cite{Dörfler1980} in 1980. 
Zhu~ \cite{MR1206546} proposed the following conjecture as a generalization of Hedetniemi's conjecture in 1992.
\begin{conjecture}\label{conj}
Let $\HH_{_{1}}=(V_{_{1}},E_{_{1}})$ and $\HH_{_{2}}=(V_{_{2}},E_{_{2}})$ be two hypergraphs. Then 
$$\chi(\HH_{_{1}}\times\HH_{_{2}})=\min\{\chi(\HH_{_{1}}),\chi(\HH_{_{2}})\}.$$
\end{conjecture}
One can easily derive a proper  coloring of $\HH_{_{1}}\times \HH_{_{2}}$ from a proper  coloring of $\HH_{_{1}}$ or of $\HH_{_{2}}$. Therefore the  hard part is  to show  that 
$\chi(\HH_{_{1}}\times\HH_{_{2}})\geq\min\{\chi(\HH_{_{1}}),\chi(\HH_{_{2}})\}.$
Let $\F$ be  a subhypergraph of $\HH_{_{1}}\times \HH_{_{2}}$ with the same vertex set and whose edge set consists of minimal hyperedges of $\HH_{_{1}}\times \HH_{_{2}}$. 
It is clear that any proper coloring of $\F$ is also a proper coloring of $\HH_{_{1}}\times \HH_{_{2}}$. 
 This observation shows that Conjecture~\ref{conj} is a generalization of Hedetniemi's conjecture.

For an integer $r$ and a hypergraph $\HH$, the {\it  $r$-colorability defect of $\HH$}, 
denoted by $\cd^{^r}(\HH)$, is the minimum number of vertices that should be removed from $\HH$ 
so that the induced hypergraph by the remaining vertices admits a proper coloring with $r$ colors. 

Let $\Z_r=\{\omega,\omega^{^2},\ldots,\omega^{^r}\}$ be a 
multiplicative cycle group of order $r$ with generator $\omega$. For  
$X=(x_{_{1}},\ldots,x_{_{n}})\in(\Z_r\cup\{0\})^{^{n}}$, a sequence $x_{_{i_{_{1}}}},\ldots,x_{_{i_{_m}}}$ 
with $1\leq i_{_{1}}<\cdots<i_{_m}\leq n$ is called an {\it  alternating subsequence of $X$} if $x_{_{i_j}}\neq 0$ 
for each $j\in[m]$ and $x_{i_j}\neq x_{_{i_{j+1}}}$ for each $j\in[m-1]$. 
The {\it  alternation number of $X$}, denoted by $\alt(X)$, is the length of the longest 
alternating subsequence of $X$. We set $\zero=(0,\ldots,0)$ and define $\alt(\zero)=0$. 
Also, for such  an $X=(x_{_{1}},\ldots,x_{_{n}})\in(\Z_r\cup\{0\})^{^{n}}$ and for  $\varepsilon\in \Z_r$, 
define $X^{^{\varepsilon}}=\left\{i:\; x_{_{i}}=\varepsilon\right\}.$ 
Note that the $r$-tuple $\big(X^\varepsilon\big)_{\varepsilon\in\Z_r}$ uniquely determines $X$ and vice versa. 
Therefore, with abuse of notations, we can write $X=\big(X^\varepsilon\big)_{\varepsilon\in\Z_r}$.  

For a hypergraph $\HH$ and a bijection $\sigma:[n]\longrightarrow V(\HH)$, the
{\it  $r$-alternation number of $\HH$ with respect to the permutation $\sigma$} is defined as follows:
$$\alt^{^r}_{_\sigma}(\HH)=\max\left\{\alt(X):\; E(\HH[\sigma(X^{^{\varepsilon}})])=\varnothing\mbox{ 
for all }\varepsilon \in\Z_p\right\}.$$ 
The {\it  $r$-alternation number} of $\HH$, denoted by $\alt^{^r}(\HH)$, is equal to 
$\min\limits_{\sigma}\alt^{^r}_{_\sigma}(\HH)$ where the minimum is taken over all bijections 
$\sigma:[n]\longrightarrow V(\HH)$ (for more details see \cite{MR3383256}). 

For any hypergraph $\HH=(V(\HH),E(\HH))$ and positive integer $r\geq 2$, the general Kneser hypergraph $\KG^r(\HH)$
 is an $r$-uniform hypergraph whose vertex set is $E(\HH)$ and whose hyperedge set   is the set of all  $r$-subsets of $E(\HH)$ containing $r$ 
 pairwise disjoint hyperedges of $\HH$. 
 Note that by this notation the well-known Kneser hypergraph $\KG^r(n,k)$ is
 the Kneser hypergraph $\KG^r\left([n],{[n]\choose k}\right)$. For $r=2$, we would rather use $\KG(\HH)$ than $\KG^r(\HH)$. 

Lov{\'a}sz in 1978, by using  tools from algebraic topology, proved that  $\chi(\KG(n,k))=n-2k+2$. 
His paper  showed an inspired and depth application of algebraic topology in combinatorics ~\cite{MR514625}.
 As a generalization of this result and to confirm a conjecture of Erd{\"o}s \cite{MR0465878},
 Alon, Frankl, and Lov{\'a}sz~\cite{MR857448} proved that the chromatic number of $\KG^r{(n,k)}$ is equal to $\lceil\frac{n-(k-1)r}{r-1}\rceil.$
A different kind of generalization of  Lov{\'a}sz's theorem has been obtained by Dol'nikov~\cite{MR953021}. He proved that  
$$\chi({\rm KG}({\mathcal H}))\geq {\cd^{^2}({\mathcal H})}.$$ 
Then, in 1992, K{\v{r}}{\'{\i}}{\v{z}}~\cite{MR1081939}
extended the both results by   Alon, Frankl, and 
Lov\'asz~\cite{MR857448} and  Dol'nikov~\cite{MR953021} by proving that
$$\chi({\rm KG}^r({\mathcal H}))\geq \left\lceil{\cd^{^r}({\mathcal H})\over r-1}\right\rceil.$$ 

Alishahi and Hajiabolhassan~\cite{MR3383256} introduced the alternation number as an 
improvement of colorability defect. 
They proved  that
$$\chi({\rm KG}^r({\mathcal H}))\geq \left\lceil{|V(\HH)|-\alt^{^r}({\mathcal H})\over r-1}\right\rceil.$$

It can be  verified that $|V(\HH)|-\alt^{^r}({\mathcal H})\geq\cd^{^r}(\HH)$ and the inequality is 
often strict~\cite{MR3383256}. Therefore, the preceding lower bound for chromatic number surpasses the 
Dol'nikov-K{\v{r}}{\'{\i}}{\v{z}} lower bound. 
Recently, Hajiabolhassan and Meunier~\cite{HaMe16} extended the  
Alishahi-Hajiabolhassan result (as well as the Dol'nikov-K{\v{r}}{\'{\i}}{\v{z}} result) to 
the categorical product of general Kneser hypergraphs as follows. 
\begin{alphtheorem}{\rm\cite{HaMe16}}\label{thm:lowerhajimeun}
Let $\HH_{_{1}},\ldots,\HH_{_{t}}$ be hypergraphs and $r$ be an integer, where $r\geq 2$.
Then 
$$\chi(\KG^r(\HH_{_{1}})\times\cdots\times\KG^r(\HH_{_{t}}))\geq\left\lceil{1\over r-1}\min\limits_{i\in[t]}
\left(|V(\HH_{_{i}})|-\alt^{^r}(\HH_{_{i}})\right)\right\rceil.$$
\end{alphtheorem}
Using Theorem~\ref{thm:lowerhajimeun}, Hajiabolhassan and Meunier introduced some new  families of hypergraphs satisfying Zhu's conjecture. 

From another point of view, Simonyi and Tardos~\cite{Simonyi&Tardos2007}  generalized the Dol'nikov result. Indeed, they proved   
 that for any hypergraph $\HH$, if $t=\cd^{^2}(\HH)$, then any proper coloring of $\KG(\HH)$ contains a complete bipartite subgraph 
$K_{\left\lfloor {t\over 2}\right\rfloor,\left\lceil {t\over 2}\right\rceil}$ such that
all vertices of this subgraph receive different colors and  these different 
colors occur alternating on the two parts 
of the bipartite graph with respect to their natural order. 
Then, this result as well as the Dol'nikov-K{\v{r}}{\'{\i}}{\v{z}} result was extended to Kneser hypergraphs by Meunier \cite{Meunier14}
as in the next theorem. A common generalization of the Simonyi-Tardos result and a result by Chen~\cite{MR2763055} can be found in~\cite{AliHajMeu2017}.

\begin{alphtheorem}\label{colorful1}
Let $\HH$ be a hypergraph and $p$ be a prime number.
Any proper coloring of $\KG^p(\HH)$ contains a 
colorful,  balanced, and complete $p$-partite subhypergraph $\F$ with  
$\cd^p(\HH)$ vertices.  
\end{alphtheorem}
It should be mentioned that, in his paper~\cite{Meunier14}, Meunier also generalized Theorem~\ref{colorful1} and proved that this theorem remains true by 
replacing $\cd^p(\HH)$ with $|V(\HH)|-\alt^{^p}(\HH)$. Moreover, several extensions of this result were presented in~\cite{Alishahi2017}. 

As an improvement of $r$-colorability defect, the equitable $r$-colorability defect was introduced in~\cite{AbyAli2017}.  
For a hypergraph $\HH$, the {\it equitable $r$-colorability defect of $\HH$}, 
denoted by $\ecd^{^r}(\HH)$, is the minimum number of vertices which should be removed so that the 
induced subhypergraph by the remaining vertices admits an equitable $r$-coloring,
i.e., an $r$-coloring in which the sizes of color classes differ by at most $1$. 
Clearly, $\ecd^{^r}(\HH)\geq \cd^{^r}(\HH)$. As a generalization of Theorem~{\ref{colorful1}}, it was proved that 
any proper coloring of $\KG^p(\HH)$ contains a colorful,  balanced, and complete $p$-partite subhypergraph $\F$ with  
$\ecd^{^p}(\HH)$ vertices. 
It is~not difficult to construct  a hypergraph $\HH$, for 
which $\ecd^{^r}(\HH)-\cd^{^r}(\HH)$ is arbitrary large. 
Surpassing the Dol'nikov-K{\v{r}}{\'{\i}}{\v{z}} lower bound,
Abyazi Sani and Alishahi~\cite{AbyAli2017}  proved  
$$\chi({\rm KG}^r({\mathcal H}))\geq \left\lceil{\ecd^{^r}({\mathcal H})\over r-1}\right\rceil.$$ 
Furthermore, they compared this lower bound with the Dol'nikov-K{\v{r}}{\'{\i}}{\v{z}} lower bound
and Alishahi-Hajiabolhassan lower bound. 
In this regard, It was shown that there is  a family of hypergraphs $\mathscr{H}$
such that for each hypergraph $\HH\in\mathscr{H}$,  
$$\chi(\KG^r(\HH))=\left\lceil{\ecd^{^r}(\HH)\over r-1}\right\rceil,$$
while $\chi(\KG^r(\HH))-\left\lceil{\cd^{^r}(\HH)\over r-1}\right\rceil$ and $\chi(\KG^r(\HH))-\left\lceil{|V(\HH)|-\alt^{^r}(\HH)\over r-1}\right\rceil$ are 
both unbounded for the hypergraphs $\HH$ in $\mathscr{H}$. 

As the main results of this paper, motivated by the preceding discussion, 
we simultaneously extend the results by Abyazi Sani and Alishahi~\cite{AbyAli2017} and by 
Hajiabolhassan and Meunier~\cite{HaMe16} to the following theorems. 
\begin{theorem}\label{thm:colorful}
Let $\HH_{_{1}},\ldots,\HH_{_{t}}$ be hypergraphs. Let $p$ be a prime number and $\eta= \min\limits_{i\in[t]}\ecd^{^p}(\HH_{_{i}})$.
Any proper coloring of $\KG^p(\HH_{_{1}})\times\cdots\times\KG^p(\HH_{_{t}})$ contains a 
colorful,  balanced, and complete $p$-partite subhypergraph $\F$ with  
$\eta$ vertices.  
\end{theorem}
\noindent{\bf Remark.} In the last section, we show that Theorem~\ref{thm:colorful} is true if we set 
$\eta= \min\limits_{i\in[t]}\left(|V(\HH_{_{i}})|-\alt^p(\HH_{_{i}})\right)$. Therefore, we have the same statement as in Theorem~\ref{thm:colorful} even if we set 
$$\eta= \max\Big\{\min\limits_{i\in[t]}\ecd^{^p}(\HH_{_{i}}),\, \min\limits_{i\in[t]}|V(\HH_{_{i}})|-\alt^{^p}(\HH_{_{i}})\Big\}.$$

Let $c$ be the proper coloring with color set $[C]$. Let $\F$ be the colorful,  balanced, and complete $p$-partite subhypergraph whose existence in ensured by Theorem~\ref{thm:colorful}.
Therefore, any color appears in at most
$p-1$ vertices of $\F$. Consequently, the previous theorem results in 
$$\chi(\KG^p(\HH_{_{1}})\times\cdots\times\KG^p(\HH_{_{t}}))\geq\left\lceil{1\over p-1}\min\limits_{i\in[t]}\ecd^{^p}(\HH_{_{i}})\right\rceil,$$ 
which can be extended for an arbitrary $r\geq 2$ as follows.
\begin{theorem}\label{thm:lower}
Let $\HH_{_{1}},\ldots,\HH_{_{t}}$ be hypergraphs and $r\geq 2$ be a positive integer, where $r\geq 2$.
Then 
$$\chi(\KG^r(\HH_{_{1}})\times\cdots\times\KG^r(\HH_{_{t}}))\geq\left\lceil{1\over r-1}\min\limits_{i\in[t]}\ecd^{^r}(\HH_{_{i}})\right\rceil.$$
\end{theorem}

\noindent{\bf Example.} In what follows, by introducing some hypergraphs, we compare two lower bounds presented in Theorems~\ref{thm:lowerhajimeun} and \ref{thm:lower}.
Let $n,k,r$ and $a$ be positive integers, where $n\geq rk$, $n> a$ and $r\geq 2$. 
Define $\HH(n,k,a)$ to be a hypergraph with the vertex set $[n]$ and the edge set $$\big\{B\colon B\subseteq [n], |B|=k, \mbox{ and } B\not\subseteq [a]\big\}.$$
Let  $\KG^r(n,k,a)$ denote the hypergraph $\KG^r(\HH(n,k,a))$. It was proved in~\cite{AbyAli2017} that if either $a\leq 2k-1$ or $a\geq rk-1$, then 
$\chi\left(\KG^r(n,k,a)\right)=\left\lceil{n-\max\{a,k-1\}\over r-1}\right\rceil.$ 
Indeed, for $a\geq rk-1$, it was proved that $$\chi\left(\KG^r(\HH(n,k,a))\right)=\left\lceil{n-\max\{a,k-1\}\over r-1}\right\rceil=\left\lceil{\ecd^{^r}(\HH(n,k,a))\over r-1}\right\rceil.$$
One should notice that the chromatic number of $\KG^r(\HH(n,k,a))$ was left open for several values of $a$ with $2k\leq a\leq rk-2$. 
Note that Theorem~\ref{thm:lower} implies the validity of Zhu's conjecture for the family of hypergraphs $\KG^r(n,k,a)$ provided that $a\geq rk-1$.
What is interesting about the hypergraph $\KG^r(\HH(n,k,a))$ is the fact that for $r\geq 4$ and $a\geq rk-1$,  
the value of $\ecd^{^r}(\HH(n,k,a))-(n-\alt^{^r}(\HH(n,k,a)))$ is unbounded. 
Thus, by the lower bound presented in Theorem~\ref{thm:lowerhajimeun}, we cannot derive that the family of hypergraphs $\KG^r(n,k,a)$ satisfies Zhu's conjecture.
On the other hand, there is a family $\mathscr{H}$ of hypergraphs (see~\cite{AbyAli2017}) such that for $\HH\in\mathscr{H}$, the value of 
$(n-\alt^{^r}(\HH(n,k,a)))-\ecd^{^r}(\HH(n,k,a))$ is unbounded. 
Hence, Theorem~\ref{thm:lowerhajimeun} and Theorem~\ref{thm:lower} introduce two somehow complementary lower bounds.

\section{Proofs}
This section is devoted to the proofs of Theorem~\ref{thm:colorful} and Theorem~\ref{thm:lower}. 
In the first subsection, we define some necessary tools which will be needed in the rest of the paper.  Although we assume that the reader has the basic knowledge 
in topological combinatorics, for more details, one can see~\cite{MR1988723}. 
\subsection{\bf Notations and Tools}
{\it A simplicial complex} is a pair $(V, K)$ where $V$  is a finite nonempty set and $K$ is 
a family of nonempty subsets of $V$ such that for each $A\in K$, if 
$\varnothing\neq B\subseteq A$, then $B\in K$. 
Respectively, the set $V$ and the family $K$ are called the {\it vertex set } and {\it  simplex set} of the simplicial complex $(V,K)$. 
For simplicity of notation and since we can assume that $V=\cup_{_{A\in K}}A$, 
with no ambiguity, we can point to a simplicial complex $(V,K)$ just by its simplex set $K$. 

Let $V$ and $W$ be two sets. We write $V\uplus W$ for the set $V\times\{1\}\cup W\times\{2\}.$
Let $K$ and $L$ be two simplicial complexes with the vertex sets $V$ and $W$, respectively. 
We define $K*L$, the join of $K$ and $L$, to be a simplicial complex with the vertex set $V\uplus W$ and the simplex set
$\{A\uplus B:A\in K, B\in L\}.$
Also, we write $K^{*n}$ instead of the $n$-fold join of $K$. 
 
 Let $p$ be a prime number.  The simplicial complex 
$\sigma^{^{p-1}}_{_{p-2}}$ is a  simplicial complex with the vertex set 
$\Z_p$ and with the simplex set consisting of all nonempty and proper subsets of
$\Z_p$. 
Note that $\left(\sigma^{^{p-1}}_{_{p-2}}\right)^{*n}$ is a simplicial complex with the vertex set 
$\Z_p\times [n]$ and $\varnothing\neq \tau\subseteq\Z_p\times [n]$ is a simplex of $\left(\sigma^{^{p-1}}_{_{p-2}}\right)^{*n}$ if and only if $|\tau\cap(\Z_p\times \{i\})|\leq p-1$ for 
each $i\in[n]$. It is clear that $\left(\sigma^{^{p-1}}_{_{p-2}}\right)^{*n}$ is a free simplicial complex where for each $\varepsilon\in\Z_p$ and $(\varepsilon',i)\in \Z_p\times [n]$, the action is defined by 
$\varepsilon\cdot(\varepsilon',i)=(\varepsilon\cdot\varepsilon',i)$. 
Let  $\tau\in\left(\sigma^{^{p-1}}_{_{p-2}}\right)^{*n}$ be a simplex. 
For each  $\varepsilon\in\Z_p$, define 
$\tau^{^{\varepsilon}} =\{(\varepsilon, j) : (\varepsilon, j)\in \tau\}$.  Also, define 
$$\ell(\tau) = p\cdot h(\tau) + |\{\varepsilon \in \mathbb{Z}_p : |\tau^{^{\varepsilon}}| > h(\tau)\}|,$$ 
where $h(\tau) = \min \limits_{\varepsilon \in \mathbb{Z}_p} |\tau^{^{\varepsilon}}|$. 
Note that each $X\in (\Z_p\cup\{0\})^n\setminus\{\zero\}$ represents a simplicial complex in $\Z_p^{*n}\subseteq\left(\sigma^{^{p-1}}_{_{p-2}}\right)^{*n}$ and vice versa.
Therefore, speaking about $h(X)$ and $\ell(X)$ is meaningful. Indeed, we have 
$$h(X)=\min\limits_{\varepsilon\in\Z_p}|X^\varepsilon|\qquad \mbox{and}\qquad \ell(X)=p\cdot h(X)+|\{\varepsilon\colon |X^\varepsilon|>h(X)\}|.$$

\subsection{Proof of Theorem~\ref{thm:colorful}}\label{proofthm1}
For simplicity of notation, assume that $\HH_{_{1}}=([n_{_{1}}],E_{_{1}}), \ldots, \HH_{_{t}}=([n_{_{t}}],E_{_{t}})$ and moreover, set $n=\sum\limits_{i=1}^t n_{_{i}}$. 
For each $X\in(\Z_p\cup\{0\})^{^{n}}\setminus \{\zero\}$, let $X(1)\in(\Z_p\cup\{0\})^{n_{_{1}}}$ 
be the  first $n_{_{1}}$ coordinates of $X$, $X(2)\in(\Z_p\cup\{0\})^{n_{_{2}}}$ be the next 
$n_{_{2}}$ coordinates of $X$, and so on, up to $X(t)\in(\Z_p\cup\{0\})^{n_{_{t}}}$ be the last 
$n_{_{t}}$ coordinates of $X$. Also, for each $j\in[t]$, define  $A_{_{j}}(X)$ to be the set of signs 
$\varepsilon\in\Z_p$ such that $X(j)^{^{\varepsilon}}$ contains at least one edge of $\HH_{_{j}}$. 
We remind that $X(j)^{^{\varepsilon}}$ is the set 
of all $i\in[n_{_{j}}]$ such that $x_{i+\sum_{j'=1}^{j-1}n_{j'}}=\varepsilon$. 
Define
$$\Sigma_{_{1}}=\Big\{X\in (\Z_p\cup\{0\})^{^{n}}\setminus \{\zero\}:\; A_{_{j}}(X)\neq \Z_p\mbox{ for at least one } j\in[t]\Big\}$$
and
$$\Sigma_{_{2}}=\Big\{X\in (\Z_p\cup\{0\})^{^{n}}\setminus \{\zero\}:\; A_{_{j}}(X)=\Z_p\mbox{ for all } j\in[t]\Big\}.$$

In what follows, we define two sign maps playing important roles in the proof. \\

\noindent{\bf Definition of $s_{_1}(-)$.}
Let $X\in \Sigma_{_{1}}$ be a vector such that $A_{_{j}}(X)\in\{\varnothing, \Z_p\}$ for each $j\in [t]$.  
Define 
	$$B_{_{j}}(X)=\left\{
	\begin{array}{ll}
	X(j) & \mbox{ if } A_{_{j}}(X)=\Z_p, \\ \\ 
	\{\varepsilon:\; X(j)^{^{\varepsilon}}\neq\varnothing\} & \mbox{ if } A_{_{j}}(X)=\varnothing \mbox{ and } h(X(j)) = 0, \\ \\
	Z(j) & \mbox{ if } A_{_{j}}(X)=\varnothing \mbox{ and }  h(X(j)) > 0,
	\end{array}\right.$$
	where 
	$$Z(j)^{^{\varepsilon}} = \left\{
	\begin{array}{ll}
	X(j)^{^{\varepsilon}} & \mbox{ if } |X(j)^{^{\varepsilon}}|= h(X(j)),\\ \\
	\varnothing &  \mbox{otherwise.}
	\end{array}\right.$$

Now, 
set $B(X)=\big(B_{_{1}}(X),\ldots,B_{_{t}}(X)\big)$ and 
$$L_{_{1}}=\ds\Big\{
B(X):\; X\in \Sigma_{_{1}}\mbox{ and } A_{_{j}}(X)\in\{\varnothing, \Z_p\} \mbox{ for all } j\in [t]
\Big\}.$$
By the action $\varepsilon\cdot B(X)=B(\varepsilon\cdot X)$, 
 $\Z_p$ clearly acts freely on $L_{_{1}}$. Let $s_{_1}:L_{_{1}}\longrightarrow \Z_p$ be an arbitrary 
$\Z_p$-equivariant map. Note that such a map can be defined by choosing one representative 
in each orbit and defining the value of the function arbitrary on this representative. \\

\noindent{\bf Definition of $s_{_2}(-)$.}
Clearly $\Z_p$ acts freely on 
$$L_{_{2}}=2^{\Z_p}\times\cdots\times 2^{\Z_p}\setminus (\{\varnothing,\Z_p\}\times \cdots\times \{\varnothing,\Z_p\})$$ 
by the action $\varepsilon\cdot(C_{_{1}},\ldots,C_{_{t}})=(\varepsilon\cdot C_{_{1}},\ldots, \varepsilon\cdot C_{_{t}})$. Similar to the definition of $s_{_1}(-)$, let  $s_{_2}:L_{_{2}}\longrightarrow \Z_p$ be an arbitrary $\Z_p$-equivariant map.

\subsubsection{\bf Defining the map $\lambda_{_{1}}$.}

Set $\alpha=n-\min\limits_{i\in[t]}\ecd^{^p}(\HH_{_{i}})+p-1$. For every $j\in[t]$,  define the map $\nu_{_{j}}$ as follows:
$$\footnotesize{\nu_{_{j}}(X)=\left\{
	\begin{array}{ll}
	|X(j)| & \mbox{ if } A_{_{j}}(X)=\Z_p \\ \\ 
	|A_{_{j}}(X)|+
	\max\Big\{\ell\big(\bar{X}(j)\big):\; \bar{X}(j)\subseteq X(j) \mbox{ and } 
	E(\HH_{_{j}}[\bar{X}(j)^{^{\varepsilon}}])=\varnothing\mbox{ for all } \varepsilon\in\Z_p \big\} & \mbox{ if } A_{_{j}}(X)\neq \Z_p.  
	\end{array}\right.}$$
	Now, let $\nu(X) = {\sum\limits_{j= 1}^t \nu_{_{j}}(X)}$.

Define the map 
$$\begin{array}{llll}
\lambda_{_{1}}:& \Sigma_{_{1}} & \longrightarrow & \Z_p\times\left\{1,\ldots,\alpha\right\}\\
		& X & \longmapsto & (s(X), \nu(X)), 
\end{array}
$$   
For defining  $s(X)$, we consider the following different cases. 
\begin{itemize}
\item If for each $j\in[t]$, we have $A_{_{j}}(X)\in\{\varnothing, \Z_p\}$, then 
$s(X)=s_{_1}\big(B(X)\big)$. 

\item If for some $j\in[t]$, we have $A_{_{j}}(X)\not\in\{\varnothing, \Z_p\}$, then set $s(X)$ to be $s_{_2}\big(A_{_{1}}(X),\ldots,A_{_{t}}(X)\big)$. 
\end{itemize}
\begin{lemma}\label{lambda1}
The map $\lambda_{_{1}}$ is a $\Z_p$-equivariant map with no $X,Y\in \Sigma_{_{1}}$  
such that $X\subseteq Y$, $\nu(X)=\nu(Y)$ and $s(X)\neq s(Y)$. 
\end{lemma}

\begin{proof}
 Clearly, $\lambda_{_{1}}$ is a $\Z_p$-equivariant map since two maps $s_{_1}$ and $s_{_2}$ are $\Z_p$-equivariant.
 	For a contradiction, suppose that  $X$ and $Y$ are two vectors in $\Sigma_{_{1}}$ such that  $X \subseteq Y$, $\nu(X) = \nu(Y)$ and $s(X)\neq s(Y)$.
 	Note that each $\nu_{_j}$ is monotone, i.e., if $X\subseteq Y$, then $\nu_{_j}(X)\leq  \nu_{_j}(Y)$. Therefore, 
	we have  $X(j) \subseteq Y(j)$ and consequently, $A_{_{j}}(X) \subseteq A_{_{j}}(Y)$ and 
$$\begin{array}{lll}
~&\{\ell(\bar{X}(j)):\; \bar{X}(j)\subseteq X(j) \mbox{ and } E(\HH_{_{j}}[\bar{X}(j)^{^{\varepsilon}}])=\varnothing \quad \forall  \varepsilon\in\Z_p \}&\subseteq\\
 	~&\left\{\ell(\bar{Y}(j)):\; \bar{Y}(j)\subseteq Y(j) \mbox{ and }  	E(\HH_{_{j}}[\bar{Y}(j)^{^{\varepsilon}}])=\varnothing\quad\forall \varepsilon\in\Z_p \right\}.&
\end{array}$$ 
The equality $v(X) = v(Y)$ along with the above discussion implies $\nu_{_j}(X) = \nu_{_j}(Y)$ for each $j$ and consequently; $A_{_{j}}(X) = A_{_{j}}(Y)$.
This observation leads us to the following cases. \\
\begin{itemize}
\item[{\rm I)}] $A_{_{j}}(X) \in \{\varnothing, \Z_p\}$ for each $j$.\\ 
Therefore, $s(X)=s_{_1}\big(B(X)\big)$. Since $A_{_{j}}(X)= A_{_{j}}(Y)$ for each $j$,  we have
$s(Y)=s_{_1}\big(B(Y)\big).$  
Consequently,  the fact that $s(X)\neq s(Y)$ implies that $B(X)\neq B(Y)$.       
Now, let $j_{_0}$ be smallest integer for which $B_{_{j_{_0}}}(X)\neq B_{_{j_{_0}}}(Y)$. 
We consider the following different cases. \\
	\begin{itemize}
		\item[{\rm 1)}]  When $A_{_{{j_{_0}}}}(X) = A_{_{{j_{_0}}}}(Y)=\Z_p$.  In view of the definition of $B_{_{j_{_0}}}(-)$, we have $X({j_{_0}})\subsetneq Y({j_{_0}})$. Therefore, the 
				       definition of $\nu_{_{j_{_0}}}$ implies that $\nu_{_{j_{_0}}}(X)<\nu_{_{j_{_0}}}(Y)$, which is~not possible. \\
		\item[{\rm 2)}]  When $A_{_{j_{_0}}}(X) = A_{_{j_{_0}}}(Y)=\varnothing$. Using $\nu_{_{j_{_0}}}(X)=\nu_{_{j_{_0}}}(Y)$, we have 
		     	   	       $\ell(X({j_{_0}}))=\ell(Y({j_{_0}})).$ 
				       Therefore,\[
				        p\cdot h(X(j_{_0}))+|\{\varepsilon:\; |X(j_{_0})^{^{\varepsilon}}|>h(X(j_{_0}))\}|=
				       p\cdot h(Y(j_{_0}))+|\{\varepsilon:\; |Y(j_{_0})^{^{\varepsilon}}|>h(Y(j_{_0}))\}|,\]
				       which clearly implies that $h(X(j_{_0}))=h(Y(j_{_0}))$ and 
				       $$|\{\varepsilon:\; |X(j_{_0})^{^{\varepsilon}}|>h(X(j_{_0}))\}|=|\{\varepsilon:\; |Y(j_{_0})^{^{\varepsilon}}|>h(Y(j_{_0}))\}|.$$ 
				       The fact that 
				       $X(j_{_0})\subseteq Y(j_{_0})$ results in 
				       $$\{\varepsilon:\; |X(j_{_0})^{^{\varepsilon}}|>h(X(j_{_0}))\}=\{\varepsilon:\; |Y(j_{_0})^{^{\varepsilon}}|>h(Y(j_{_0}))\}.$$ 
				       Therefore, in view of the definition of 
				       $B(-)$, we have $B_{_{j_{_0}}}(X)=B_{_{j_{_0}}}(Y)$ which is a contradiction. 
	\end{itemize}

\item[{\rm II)}] $\varnothing \subsetneq  A_{_{j}}(X)\subsetneq \Z_p$ for some $j\in[t]$. Since $s(X)\neq s(Y)$, we have
$$s_{_2}\big(A_{_{1}}(X),\ldots,A_{_{t}}(X)\big)\neq s_{_2}\big(A_{_{1}}(Y),\ldots,A_{_{t}}(Y)\big).$$  
Consequently, we must have $(A_{_{1}}(X),\ldots,A_{_{t}}(X)\big)\neq(A_{_{1}}(Y),\ldots,A_{_{t}}(Y)\big).$ Therefore, there is at least one $j$ for which 
$A_j(X)\neq A_j(Y)$ which is~not possible. 
\end{itemize} 
 \end{proof}

\subsection{Defining the map $\lambda_{_{2}}$}
Let $c$ be a proper coloring of $\KG^r(\HH_{_{1}})\times\cdots\times \KG^r(\HH_{_{t}})$
with color set $[C]$.  
For each $X\in\Sigma_{_{2}}$ and each $\varepsilon\in\Z_p$, define 
$$E^{^{\varepsilon}}(X)=\Big\{(e_{_{1}},\ldots,e_{_{t}})\in E_{_{1}}\times\cdots\times E_{_{t}}:\; e_{_{j}}\subseteq X(j)^{^{\varepsilon}}\mbox{ for each }\varepsilon\in\Z_p\Big\}.$$
Note that, in view of the definition of $\Sigma_{_{2}}$, for each $\varepsilon\in\Z_p$, we have $E^{^{\varepsilon}}(X)\neq\varnothing$. 
Now, set 
$\tau{(X)}\in(\sigma_{_{p-2}}^{^{p-1}})^{*C}$ to be a simplex defined as follows: 
$$\tau{(X)}=\Big\{(\varepsilon, c(u)):\; u=(e_{_{1}},\ldots,e_{_{t}})\in E^{^{\varepsilon}}(X)\Big\}.$$
Since $c$ is a proper coloring and $E^{^{\varepsilon}}(X)\neq\varnothing$ for each 
$\varepsilon\in\Z_p$, one can check that $\tau{(X)}$ is a simplex in 
$(\sigma_{_{p-2}}^{^{p-1}})^{*C}$ with $h(\tau{(X)})>0$, and consequently, $\ell(\tau{(X)})\geq p$.

For a positive integer $b \in [C]$, let $U_{_{b}}$ be the set consisting of all simplices 
$\tau \in \left(\sigma^{^{p-1}}_{_{p-2}}\right)^{*C}$ such that $|\tau^{^{\varepsilon}}| \in \{0, b\}$ for each 
$\varepsilon \in \mathbb{Z}_p$. 
Define $U=\bigcup\limits _{b = 1}^C U_{_{b}}$. 
Choose an arbitrary $\mathbb{Z}_p$-equivariant map $s_{_3} : U \rightarrow \mathbb{Z}_p$. Also, for each $\tau \in \left(\sigma^{^{p-1}}_{_{p-2}}\right)^{*C}$ with 
$h=h(\tau)=\min |\tau^{^{\varepsilon}}|$, define 
$$\overline{\tau}=\bigcup\limits_{\varepsilon: |\tau^{^{\varepsilon}}|=h}\tau^{^{\varepsilon}}.$$
Note that $\overline{\tau}$ is a sub-simplex of $\tau$ which is in $U$. Therefore, $s_{_3}(\overline{\tau})$ is defined.

Define the map 
$$\begin{array}{llll}
\lambda_{_{2}}:& \Sigma_{_{2}} & \longrightarrow & \Z_p\times\left\{\alpha+1,\ldots, \alpha-p+1+\max\limits_{X \in \Sigma_{_{2}}}\ell(\tau(X))\right\}\\
		& X & \longmapsto & (s_{_3}(\overline{\tau{(X)}}), \alpha-p+1+\ell(\tau(X))). 
\end{array}
$$
\begin{lemma}\label{lambda2}
The map $\lambda_{_{2}}$ is a $\Z_p$-equivariant map with no $X,Y\in \Sigma_{_{2}}$  
such that $X\subseteq Y$, 
$\lambda_{_{2}}(X)=(\varepsilon,i)$ and $\lambda_{_{2}}(Y)=(\varepsilon',i)$, where $\varepsilon\neq \varepsilon'$. 

\end{lemma}

\begin{proof}
	Obviously, $\lambda_{_{2}}$ is a $\Z_p$-equivariant map. Suppose for a contradiction 
	that $X$ and $Y$ are two vectors in $\Sigma_{_{2}}$ such that $X \subseteq Y$, $\lambda_{_{2}}(X)=(\varepsilon,i)$ and $\lambda_{_{2}}(Y)=(\varepsilon',i)$, where $\varepsilon\neq \varepsilon'$. 
	In view of the definition of $\lambda_2$, we have $\ell(\tau(X)) = \ell(\tau(Y))=i-\alpha+p-1.$ 
	Using the definition of $\ell(-)$, we must have $h(\tau(X)) = h(\tau(Y))$. Since $\tau(X)\subseteq \tau(Y)$, it implies that 
	$\overline{\tau(X)} = \overline{\tau(Y)}$ which yields  the equality $\varepsilon=s_{_3}(\overline{\tau(X)}) = s_{_3}( \overline{\tau(Y)})=\varepsilon'$, a contradiction. 
\end{proof}

\begin{lemma}\label{lem:colorful}
If there is an $X\in \Sigma_{_{2}}$  with $\ell(\tau{(X)})\geq q$, then 
$\KG^{^p}(\HH_{_{1}})\times\cdots\times\KG^{^p}(\HH_{_{t}})$ contains a colorful, balanced, and
complete $p$-partite subhypergraph with $q$ vertices.
\end{lemma}
\begin{proof}
Let $X\in \Sigma_{_{2}}$ be a vector for which we have $\ell(\tau{(X)})\geq q$. 
Let $\tau\subseteq \tau{(X)}$ be a sub-simplex such that $\ell(\tau)=|\tau|=q$.
For each $i\in[p]$, set $S_{_{i}}=\{j\in[C]:\; (\omega^{{i}},j)\in\tau\}$. 
First note that $\lfloor{q\over p}\rfloor\leq|S_{_{i}}|\leq\lceil{q\over p}\rceil$ for each $i\in[p]$.
Moreover, it is clear that $\sum\limits_{i=1}^p|S_{_{i}}|=q$. 
For each $i\in[p]$ and $s\in S_{_{i}}$, in view of the definitions of 
$\tau{(X)}$ and $S_{_{i}}$, there is a vertex $v_{_{i,s}}=(e_{_{i,1}}^{^{s}},\ldots,e_{_{i,t}}^{^{s}})$ of 
$\KG^p(\HH_{_{1}})\times\cdots\times\KG^p(\HH_{_{t}})$ such that $c(v_{_{i,s}})=s$ and 
$e_{_{i,j}}^{^{s}}\subseteq X(j)^{{\omega^i}}$ for each $j\in[t]$. Now, for $i\in[p]$, set $V_{_{i}}=\{v_{_{i,s}}:\; s\in S_{_{i}}\}$. 
Clearly, $\KG^p(\HH_{_{1}})\times\cdots\times\KG^p(\HH_{_{t}})[V_{_{1}},\ldots,V_{_{p}}]$ contains the 
desired subhypergraph. 
\end{proof}
For completing the proof of Theorem~\ref{thm:colorful}, we need to use a generalization of the Borsuk-Ulam theorem  by  Dold, see~\cite{MR711043, MR1988723}. 
Indeed, Dold's theorem implies that if there is a simplicial $\Z_p$-map from a simplicial $\Z_p$-complex $K_{_1}$ to a free simplicial $\Z_p$-complex $K_{_2}$, then 
the dimension of $K_{_2}$ should be strictly larger than the connectivity of $K_{_1}$.

\begin{proof}[\bf Completing the proof of Theorem~\ref{thm:colorful}]
For simplicity of notation, let 
$$m=\alpha-p+1+\max\limits_{X\in \Sigma_{_{2}}}\ell(\tau{(X)}).$$
In view of Lemma~\ref{lem:colorful}, it suffices to show that 
 $$\max\limits_{X\in \Sigma_{_{2}}}\ell(\tau{(X)})\geq \min\limits_{i\in[t]}\ecd^{^p}(\HH_{_{i}}).$$
To this end, define 
$\lambda:  (\Z_p\cup\{0\})^{^{n}}\setminus \{\zero\}  \longrightarrow  \Z_p\times[m]$ 
such that for each $X\in (\Z_p\cup\{0\})^{^{n}}\setminus \{\zero\}$, if $X\in\Sigma_{_{1}}$, then $\lambda(X)=\lambda_{_{1}}(X)$,
otherwise $\lambda(X)=\lambda_{_{2}}(X)$. 
In view of Lemma~\ref{lambda1} and Lemma~\ref{lambda2},  $\lambda(-)$ is a $\Z_p$-equivariant simplicial map from $\sd(\Z_p^{*n})$ to $\Z_p^{*m}$.
Consequently, according to  Dold's theorem, the dimension of $\Z_p^{*m}$ should be strictly larger than the connectivity of $\sd(\Z_p^{*n})$, that is $m\geq n$
as desired. 
\end{proof}

\subsection{Proof of Theorem~\ref{thm:lower}}
To prove Theorem~\ref{thm:lower}, we introduce a reduction reducing this theorem to the prime case of $r$ which is known to be true by the discussion right after Theorem~\ref{thm:colorful}.
One should notice that this reduction is a refinement of the well-known reduction originally due to 
K\v{r}\'{\i}\v{z}~\cite{MR1665335,Kriz2000}, which has been used in some other papers as well, for instance see~\cite{MR3383256,HaMe16,Ziegler2002,Ziegler2006}. 
In what follows, we use a similar approach as in~\cite{HaMe16}.

\begin{lemma}\label{reduction}
Let $r'$ and $r''$ be two positive integers. If Theorem~\ref{thm:lower} holds for both $r'$ and $r''$, then it holds also for $r=r'r''$.
\end{lemma} 
For two positive integers $s$ and $C$ and a hypergraph $\HH$, define a new hypergraph $\mathcal{T}_{\HH,C,s}$ as follows:
$$
\begin{array}{lll}
V(\mathcal{T}_{\HH,C,s})&=&V(\HH) \\ 
E(\mathcal{T}_{\HH,C,s})&=&\Big\{A\subseteq V(\HH):\ecd^{s}(\HH[A])>(s-1)C\Big\}.
\end{array}$$

The following lemma  can be proved with a similar approach as in ~\cite[Lemma~3]{HaMe16}.
\begin{lemma}
Let $r$ and $s$ be two positive integers. Then
$$\ecd^{^{rs}}(\HH)\leq r(s-1)C+\ecd^{^r}(\mathcal{T}_{\HH,C,s})$$
\end{lemma}
\begin{proof}[Proof of Lemma~\ref{reduction}]
Using the previous lemma instead of Lemma~3 in the proof of Lemma~1 in~\cite{HaMe16} leads us to the proof.
\end{proof}

\section{Concluding Remarks}
Although there are hypergraphs $\HH$ for which  $\ecd^{^r}(\HH)-(|V(\HH)|-\alt^{^r}(\HH))$ is arbitrary large, one can construct hypergraphs 
$\HH$ making $(|V(\HH)|-\alt^{^r}(\HH))-\ecd^{^r}(\HH)$ arbitrary large, see~\cite{AbyAli2017}.    
Therefore, it might be interesting to have a statement similar to Theorem~\ref{thm:colorful} 
using $|V(\HH_{_i})|-\alt^{^p}(\HH_{_i})$ instead of $\ecd^{^p}(\HH_i)$. 
Note that such a statement generalizes Theorem~\ref{thm:lowerhajimeun} as well. 
To prove this statement, we need to slightly modify the proof of Theorem~\ref{thm:colorful} as follows.
\begin{itemize}
\item Throughout Section~\ref{proofthm1} , we replace  $\min\limits_{i\in[t]}\ecd^{^p}(\HH_{_{i}})$ by $\min\limits_{i\in[t]}\left(|V(\HH_{_i})|-\alt^{^p}(\HH_{_{i}})\right)$. 
\item In the definition of $\lambda_1(X)$, we use $\alt(-)$ instead of  function $\ell(-)$ to define $\nu_{_j}(X)$'s.\\ 

\item For any $X$ with $A_{_{j}}(X)\in\{\varnothing, \Z_p\}$ for each $j\in[t]$, in the definition of $\lambda_1(X)$, 
	we set $s(X)$ to be the first nonzero entry of $X$. 
\end{itemize}
With the same approach as in Section~\ref{proofthm1}, it is straightforward to check that 
Lemmas~\ref{lambda1},~\ref{lambda2}~and~\ref{lem:colorful} are still valid with the preceding modifications. 
Therefore, again applying Dold's theorem leads us to the following statement. 
\begin{theorem}\label{thm:colorfulalt}
Let $\HH_{_{1}},\ldots,\HH_{_{t}}$ be hypergraphs and $\eta=\min\limits_{i\in[t]}\left(|V(\HH_{_i})|-\alt^{^p}(\HH_{_i})\right)$ where $p$ is a prime number.
Any proper coloring of $\KG^p(\HH_{_{1}})\times\cdots\times\KG^p(\HH_{_{t}})$ contains a 
colorful,  balanced, and complete $p$-partite subhypergraph $\F$ with  
$\eta$ vertices.  
\end{theorem}
Also, the question of whether Theorem~\ref{thm:colorful} and Theorem~\ref{thm:colorfulalt} hold for  an arbitrary positive integer $r$ instead of a prime number $p$  is interesting.

\bibliographystyle{plain}
\def\cprime{$'$} \def\cprime{$'$}

\end{document}